\documentclass[32pt,twoside]{amsart}
\usepackage{graphicx}
\usepackage{amssymb,amsfonts,amstext,amsmath,amsthm}
\newcommand{\ben}{\begin{enumerate}}
\newcommand{\een}{\end{enumerate}}
\newtheorem{thm}{Theorem}[section]

\newtheorem{lem}[thm]{Lemma}
\newtheorem{rem}[thm]{Remark}
\newtheorem{exm}[thm]{Example}
\newtheorem{cor}[thm]{Corollary}

\input{amssym.def}
\input{amssym.tex}
\thanks{$^1$Present Address: Air Force Institute of Technology, Kaduna, Nigeria}
\thanks{$^2$Present Address: Federal University of Agriculture, Makurdi, Nigeria}
\begin{document}

\title{$C_4$ and $C_6$ decomposition of the tensor product of complete graphs}
\author[O. Oyewumi and A. D. Akwu] {$^1$Opeyemi Oyewumi and $^2$Abolape Deborah Akwu}
\address{Opeyemi Oyewumi \hfill\break\indent General Studies Department, Air Force Institute of Technology, Kaduna, Nigeria}
\email{opeyemioluwaoyewumi@gmail.com}
\address{Abolape Deborah Akwu\hfill\break\indent Department of Mathematics, Federal University of Agriculture, Makurdi, Nigeria}
\email{abolaopeyemi@yahoo.co.uk }
%\label{pageinit}
\date{}

\begin{abstract} Let $G$ be a simple and finite graph. A graph is said to be \textit{decomposed} into subgraphs $H_1$ and $H_2$ which is denoted by $G= H_1 \oplus H_2$, if $G$ is the edge disjoint union of $H_1$ and $H_2$. If $G= H_1 \oplus H_2 \oplus H_3 \oplus \cdots \oplus H_k$, where\\$H_1$,$H_2$,$H_3$, ..., $H_k$ are all isomorphic to $H$, then $G$ is said to be $H$-decomposable. Futhermore, if $H$ is a cycle of length $m$ then we say that $G$ is $C_m$-decomposable and this can be written as $C_m|G$. Where $ G\times H$ denotes the tensor product of graphs $G$ and $H$, in this paper, we prove the necessary and sufficient conditions for the existence of $C_4$-decomposition (respectively, $C_6$-decomposition ) of $K_m \times K_n$. Using these conditions it can be shown that every even regular complete multipartite graph $G$ is  $C_4$-decomposable (respectively, $C_6$-decomposable) if the number of edges of $G$ is divisible by $4$ (respectively, $6$).  
\end{abstract}
\maketitle
\section{\bf Introduction}
Let $C_m$, $K_m$ and $K_m -I$ denote cycle of length $m$, complete graph on $m$ vertices and complete graph on $m$ vertices minus a $1$-factor respectively. By an $m$-cycle we mean a cycle of length $m$. Let $K_{n,n}$ denote the complete bipartite graph with $n$ vertices in each bipartition set and $K_{n,n}-I$ denote $K_{n,n}$, with a $1$-factor removed.  All graphs considered in this paper are simple and finite. A graph is said to be \textit{decomposed} into subgraphs $H_1$ and $H_2$ which is denoted by $G= H_1 \oplus H_2$, if $G$ is the edge disjoint union of $H_1$ and $H_2$. If $G= H_1 \oplus H_2 \oplus \cdots \oplus H_k$, where $H_1$,$H_2$, ..., $H_k$ are all isomorphic to $H$, then $G$ is said to be $H$-decomposable. Futhermore, if $H$ is a cycle of length $m$ then we say that $G$ is $C_m$-decomposable and this can be written as $C_m|G$. A $k$-factor of $G$ is a $k$-regular spanning subgraph. A $k$-factorization of a graph $G$ is a partition of the edge set of $G$ into $k$-factors. A $C_k$-factor of a graph is a $2$-factor in which each component is a cycle of length $k$. A \textit{resolvable} $k$-cycle decomposition (for short $k$-RCD) of $G$ denoted by $C_k||G$, is a $2$-factorization of $G$ in which each $2$-factor is a $C_k$-factor. 

For two graphs $G$ and $H$ their tensor product $G \times H$ has vertex set $V(G) \times V(H)$ in which two vertices $(g_1,h_1)$ and $(g_2,h_2)$ are adjacent whenever $g_1g_2 \in E(G)$ and $h_1h_2 \in E(H)$. From this, note that the tensor product of graphs is distributive over edge disjoint union of graphs, that is if $G= H_1 \oplus H_2 \oplus \cdots \oplus H_k$, then $G \times H = (H_1 \times H)\oplus (H_2 \times H) \oplus \cdots\oplus(H_k \times H)$. Now, for $h \in V(H)$, $V(G) \times h = \{(v,h)|v \in V(G)\}$ is called a \textit{column} of vertices of $G \times H$ corresponding to $h$. Further, for  $y \in V(G)$, $y \times V(H) = \{(y,v)|v \in V(H)\}$ is called a \textit{layer} of vertices of $G \times H$ corresponding to $y$. It is true that $K_m \times K_2$ is isomorphic to the complete bipartite graph $K_{m,m}$ with the edges of a perfect matching removed, i.e. $K_m \times K_2 \cong K_{m,m}-I$, where $I$ is a $1$-factor of $K_{m,m}$.   \\
The Lexicographic product $G \ast H$ of two graphs $G$ and $H$ is the graph having the vertex set $V(G) \times V(H)$, in which two vertices $(g_1,h_1)$ and $(g_2,h_2)$ are adjacent if either $g_1,g_2 \in E(G)$; or $g_1 = g_2$ and $h_1,h_2 \in E(H)$. \\
The problem of finding $C_k$-decomposition of $K_{2n+1}$ or $K_{2n}-I$ where $I$ is a $1$-factor of $K_{2n}$, is completely settled by Alspach, Gavlas and \v Sajna in two different papers (see \cite{AG,S}). A generalization to the above complete graph decomposition problem is to find a $C_k$-decomposition of $K_m \ast \overline K_n$, which is the complete $m$-partite graph in which each partite set has $n$ vertices. The study of cycle decompositions of $K_m \ast \overline K_n$ was initiated by Hoffman et al. \cite{HCC}. In the case when $p$ is a prime, the necessary and sufficient conditions for the existence of $C_p$-decomposition of $K_m \ast \overline K_n$, $p\geq 5$ is obtained by Manikandan and Paulraja in \cite{MP1, MP2, MP4}. Billington \cite{B1} has studied the decomposition of complete tripartite graphs into cycles of length 3 and 4. Furthermore, Cavenagh and Billington \cite{CB1} have studied $4$-cycle, $6$-cycle and $8$-cycle decomposition of complete multipartite graphs. 
Billington et al. \cite{B2} have solved the problem of decomposing $(K_m \ast \overline K_n)$ into $5$-cycles. Similarly, when $p \geq 3$ is a prime, the necessary and sufficient conditions for the existence of $C_{2p}$-decomposition of $K_m \ast \overline K_n$ is obtained by Smith (see \cite{BRS1}). 
For a prime $p \geq 3$, it was proved in \cite{BRS2} that $C_{3p}$-decomposition of $K_m \ast \overline K_n$ exists if the obvious necessary conditions are satisfied. 
As the graph $K_m \times K_n \cong K_m \ast \overline K_n - E(nK_m)$ is a proper regular spanning subgraph of $K_m \ast \overline K_n$. 
It is therefore natural to think about the cycle decomposition of $K_m \times K_n$.

The results in \cite{MP1, MP2, MP4} also gives the necessary and sufficient conditions for the existence of a $p$-cycle decomposition, (where $p \geq 5$ is a prime number) of the graph $K_m \times K_n$. In \cite{MP3} it was shown that the tensor product of two regular complete multipartite graph is Hamilton cycle decomposable. Muthusamy and Paulraja in \cite{MP} proved the existence of $C_{kn}$-factorization of the graph $C_k \times K_{mn}$, where $mn \neq 2 (\textrm{mod}\ 4)$ and $k$ is odd. Paulraja and Kumar \cite{PK} showed that the necessary conditions for the existence of a resolvable $k$-cycle decomposition of tensor product of complete graphs are sufficient when $k$ is even.
In this paper, we prove the necessary and sufficient conditions for $K_m \times K_n$, where $m,n \geq 2$, to have a $C_4$-decomposition (respectively, $C_6$-decomposition). Among other results, here we prove the following main results.
It is not surprising that the conditions in Theorem \ref{TT:O1} and Theorem \ref{T:O1}  are "symmetric" with respect to $m$ and $n$ since $K_m \times K_n \cong K_n \times K_m$. 
\begin{thm}\label{TT:O1}
For $m,n \geq 2$, $C_4| K_m \times K_n$ if and only if either
\ben  
\item  
 $n \equiv 0 \ (\textrm{mod}\ 4)$ and $m$ is odd,
\item
 $m \equiv 0 \ (\textrm{mod}\ 4)$ and $n$ is odd or
 \item
 $m \ or \ n \equiv 1 \ (\textrm{mod}\ 4)$
\een
\end{thm} 
\begin{thm}\label{T:O1}
For $2 \leq m,n, C_6| K_m \times K_n$ if and only if 
 $m \equiv 1 \ or \ 3\ (\textrm{mod}\ 6)$ or $n \equiv 1 \ or \ 3\ (\textrm{mod}\ 6)$. 
\end{thm}
\begin{thm}\label{T:O2}
Let $m$ be an even integer and $m \geq 6$, then $C_6| K_m - I \times K_n$ if and only if $m \equiv 0 \ or \ 2\ (\textrm{mod}\ 6)$ 
\end{thm}
\section{Preliminaries}
\subsection{Definition}Let $\rho$ be a permutation of the vertex set $V$ of a graph
$G$. For any subset $U$ of $V$, $\rho$ acts as a function from $U$ to $V$ by considering
the restriction of $\rho$ to $U$. If $H$ is a subgraph of $G$ with vertex set $U$, then
$\rho(H)$ is a subgraph of $G$ provided that for each edge $xy \in E(H)$,
$\rho(x)\rho(y) \in E(G)$. In this case, $\rho(H)$ has vertex set $\rho(U)$ and edge set
$\{\rho(x) \rho(y): xy \in E(H)\}$.\\
\ \\
\indent
Next, we give some existing results on cycle decomposition of complete graphs. 
\begin{thm}\label{TT:O5}\cite{KK}
Let $m$ be an odd integer and $m \geq 3$. If $m \equiv 1 \ or \ 3\ (\textrm{mod}\ 6)$ then $C_3|K_m$. 
\end{thm}
\begin{thm}\label{TT:O6}\cite{S}
Let $n$ be an odd integer and $m$ be an even integer with $3 \leq m \leq n$. The graph $K_n$ can be decomposed into cycles of length $m$ whenever $m$ divides the number of edges in $K_n$.
\end{thm} 
Now we have the following lemma, this lemma gives the cycle decomposition of the complete graph $K_m$ into cycles of length $3$ and $4$. 
\begin{lem}\label{L:O1}
For $m \equiv 5(\textrm{mod}\ 6)$, there exist positive integers $p$ and $q$ such that $K_m$ is decomposable into $p$ $3$-cycles and $q$ $4$-cycles. 
\end{lem} 
\begin{proof}
Let the vertices of $K_m$ be $0,1,...,m-1$. The $4$-cycles are $(i,i+1+2s,i-1,i+2+2s), s=0,1,...,(m-i)/2-2, i=1,3,...,m-4$. The $3$-cycles are $(m-1,i-1,i), i=1,3,...,m-2$. Hence the proof.  
\end{proof}
\indent
The following theorem is on the complete bipartite graph minus a $1$-factor, it was obtained by Ma et. al \cite{MA}. 
\begin{thm}\label{MA}\cite{MA}
Let $m$ and $n$ be positive integers. Then there exist an $m$ cycle system of $K_{n,n}-I$ if and only if $n \equiv 1\ (\textrm{mod}\ 2)$, $m \equiv 0 \ (\textrm{mod}\ 2)$, $4 \leq m \leq 2n$ and $n(n-1) \equiv 0\ (\textrm{mod}\ m)$.
\end{thm}
\indent
From the theorem above we have the following corollary.
\begin{cor}\label{MAcor}
The graph $K_{n,n}-I$, where $I$ is a $1$-factor of $K_{n,n}-I$ admits a $C_4$ decomposition if and only if $n \equiv 1 \ (\textrm{mod}\ 4)$.
\end{cor}
\indent
The following result is on the complete bipartite graphs.
\begin{thm}\label{TT:O3}\cite{SO}
The complete bipartite graph $K_{a,b}$ can be decomposed into cycles of length $2k$ if and only if $a$ and $b$ are even, $a\geq k$, $b\geq k$ and $2k$ divides $ab$. 
\end{thm}
\indent
The next result is on cycle decomposition of complete graph minus a $1$-factor.
\begin{thm}\label{T:O5}\cite{S}
 Let $n$ be an even integer and $m$ be an odd integer with $3 \leq m \leq n$. 
 The graph $K_n- I$ can be decomposed into cycles of length $m$ whenever $m$ divides the number of edges in $K_n - I$.   
\end{thm}      
\section{ \bf $C_4$ Decomposition of $C_m \times K_n$}
\noindent We begin this section with the following lemma. 
\begin{lem}\label{TT:O2}
$C_4|C_3 \times K_4$.
\end{lem} 
\begin{proof} Following from the definition of the tensor product of graphs, let $U^1= \{u_1,v_1,w_1\}$, $U^2= \{u_2,v_2,w_2\} $,$...$, $U^4=\{u_4,v_4,w_4\}$ form the partite sets of vertices in the product $C_3 \times K_4$. For $1\leq i,j\leq 4$, surely $U^i \cup U^j$, $i\neq j$ induces a $K_{3,3}-I$, where $I$ is a $1$-factor of $K_{3,3}$. \\
 A $C_4$ decomposition of $C_3 \times K_4$ is given below:
 \\$\{u_1,v_4,u_2,w_3\}$,$\{u_1,v_3,u_4,w_2\}$,$\{u_1,v_2,u_3,w_4\}$,$\{u_2,v_3,w_2,v_1\}$,$\{u_3,v_1,w_3,v_4\}$,\\$\{u_2,w_1,v_3,w_4\}$,$\{u_3,w_2,v_4,w_1\}$,$\{u_4,v_1,w_4,v_2\}$ and $\{u_4,w_1,v_2,w_3\}$
\end{proof}
\noindent
Next, we have the following lemma which follows from Lemma \ref{TT:O2}.
\begin{lem}\label{TT:O4}
$C_4|C_3 \times K_5$.
\end{lem}
\begin{proof} 
Suppose we fix the $4$-cycles already given in Lemma \ref{TT:O2}, clearly the graph which
remains after removing the edges of $C_3 \times K_4$ from $C_3 \times K_5$ can be decomposed into $3$ copies of $K_{2,4}$. Now, by Theorem \ref{TT:O3} the graph $K_{2,4}$ can be decomposed into cycles of length $4$. Hence $C_4|C_3 \times K_5$.  
\end{proof}
\indent
The following theorem is an extension of Lemma \ref{TT:O2} and Lemma \ref{TT:O4}.   
\begin{thm}\label{TT:O7}
$C_4|C_3 \times K_n$ if and only if $n \equiv 0 \ or \ 1\ (\textrm{mod}\ 4)$.
\end{thm}
\begin{proof}  
Suppose that $C_4|C_3 \times K_n$. The graph $C_3 \times K_n$ has $3n(n-1)$ edges. For $C_4|C_3 \times K_n$ it implies that $n(n-1) \equiv 0 \ (\textrm{mod}\ 4)$. Hence $n \equiv 0 \ or \ 1\ (\textrm{mod}\ 4)$.\\ Following the definition of tensor product of graphs, let  $U^1=\{u_1,v_1,w_1\}$, $U^2=\{u_2,v_2,w_2\}$,..., $U^n=\{u_n,v_n,w_n\}$ form the partite sets of vertices in the product $C_3 \times K_n$. For $1\leq i,j\leq n$, surely $U^i \cup U^j$, $i\neq j$ induces a $K_{3,3}-I$, where $I$ is a $1$-factor of $K_{3,3}$.\\ Next, we prove the sufficiency in two cases.\\
{\bf Case 1.} Whenever $n \equiv 0 (\textrm{mod}\ 4)$. Let $n=4t$ where $t\geq 1$. \\ 
Next we note that $C_3 \times K_n \cong (C_3 \times K_4) \ + \ (C_3 \times K_4) \ + \ (C_3 \times K_4) \ + \ \cdots + \ (C_3 \times K_4) \ + \ H^*$, $H^*$ is the graph containing the edges of $C_3 \times K_n$ which are not covered by these $t$ copies of $C_3 \times K_4$. By Lemma \ref{TT:O2} the product $C_3 \times K_4$ admits a $C_4$-decomposition.  \\
Furthermore, we define the set $U=\{u^1,u^2,...,u^p\}$,$V=\{v^1,v^2,...,v^p\}$ and $W=\{w^1,w^2,...,w^p\}$ where
$p=n/4$ and for $j=1,2,...,p$, $u^j=\{u_i|4j-3\leq i \leq 4j\}$, $v^j=\{v_i|4j-3\leq i \leq 4j\}$ and $w^j=\{w_i|4j-3\leq i \leq 4j\}$.\\
Now, $H^*$ is decomposable into graphs isomorphic to $K_{4,4n-4}$. Indeed, the $K_{4,4n-4}$ graphs in the decomposition of $H^*$ are induced by $(u^i \cup v^1 \cup v^2 \cup \cdots \cup v^p)\setminus v^i$, $(u^i \cup w^1 \cup w^2 \cup \cdots \cup w^p)\setminus w^i$ and $(v^i \cup w^1 \cup w^2 \cup \cdots \cup w^p)\setminus w^i$, $i=1,2,...,p$. By Theorem \ref{TT:O3} $C_4|K_{4,4n-4}$. Therefore we have decomposed $C_3 \times K_n$ into $4$-cycles when $n \equiv 0 (\textrm{mod}\ 4)$.\\

{\bf Case 2.} Whenever $n \equiv 1 (\textrm{mod}\ 4)$. Let $n=4t+1$ where $t\geq 1$.\\  
By removing $U^1$, we obtain a copy of $C_3 \times K_{n-1}$, so we may apply Case 1. The remaining structure can be decomposed into $3K_{2,4t}$ and by Theorem \ref{TT:O3} $C_4|K_{2,4t}$. Therefore  $C_4|C_3 \times K_n$ when $n \equiv 1 (\textrm{mod}\ 4)$.    
\end{proof}

\section{ \bf $C_6$ Decomposition of $C_3 \times K_n$}
 \begin{thm}\label{T:O3}
 For all $n$, $C_6|C_3 \times K_n$.
 \end{thm}
 \begin{proof} Following from the definition of tensor product of graphs, let $U^1= \{u_1,v_1,w_1\}$, $U^2= \{u_2,v_2,w_2\} $,$...$, $U^n=\{u_n,v_n,w_n\}$ form the partite set of vertices in $C_3 \times K_n$. Also, $U^i$ and $U^j$ has an edge in $C_3 \times K_n$ for $1\leq i,j\leq n$ and $i\neq j$ if the subgraph induce $K_{3,3}-I$, where $I$ is a $1$-factor of $K_{3,3}$. Now, each subgraph $U^i \cup U^j$  is isomorphic to $K_{3,3}-I$. But $K_{3,3}-I$ is a cycle of length six. Hence the proof.
\end{proof}
\begin{exm}
The graph $C_3 \times K_7$ can be decomposed into cycles of length $6$.\\
{\bf Solution.}
Let the partite sets (layers) of the tripartite graph $C_3 \times K_7$ be $U=\{u_1, u_2, . . . , u_7\}$, $V=\{v_1, v_2, . . . , v_7\}$ and $W=\{w_1,w_2, . . . ,w_7\}$. We assume that the vertices of $U,V$ and $W$ having same subscripts are the corresponding vertices of the partite sets. A  $6$-cycle decomposition of $C_3 \times K_7$ is given below:
\\$\{u_1,v_2,w_1,u_2,v_1,w_2\}$,$\{u_1,v_3,w_1,u_3,v_1,w_3\}$,$\{u_2,v_3,w_2,u_3,v_2,w_3\}$,\\$\{u_1,v_4,w_1,u_4,v_1,w_4\}$,$\{u_2,v_4,w_2,u_4,v_2,w_4\}$,$\{u_3,v_4,w_3,u_4,v_3,w_4\}$,\\$\{u_1,v_5,w_1,u_5,v_1,w_5\}$,$\{u_2,v_5,w_2,u_5,v_2,w_5\}$,$\{u_3,v_5,w_3,u_5,v_3,w_5\}$,\\$\{u_4,v_5,w_4,u_5,v_4,w_5\}$,$\{u_1,v_6,w_1,u_6,v_1,w_6\}$,$\{u_2,v_6,w_2,u_6,v_2,w_6\}$,\\$\{u_3,v_6,w_3,u_6,v_3,w_6\}$,$\{u_4,v_6,w_4,u_6,v_4,w_6\}$,$\{u_5,v_6,w_5,u_6,v_5,w_6\}$,\\$\{u_1,v_7,w_1,u_7,v_1,w_7\}$,$\{u_2,v_7,w_2,u_7,v_2,w_7\}$,$\{u_3,v_7,w_3,u_7,v_3,w_7\}$,\\$\{u_4,v_7,w_4,u_7,v_4,w_7\}$,$\{u_5,v_7,w_5,u_7,v_5,w_7\}$,$\{u_6,v_7,w_6,u_7,v_6,w_7\}$.
\end{exm} 

\section{ \bf $C_m$ Decomposition of $C_m \times K_n$}
Next, we establish the following lemma.
\begin{lem}\label{TT:O8}
For all $n \geq 3$, $C_4|C_4 \times K_n$.
\end{lem} 
\begin{proof}
From the definition of tensor product of graphs, let $U^1= \{u_1,v_1,w_1,x_1\}$, $U^2= \{u_2,v_2,w_2,x_2\} $,$...$, $U^n=\{u_n,v_n,w_n,x_n\}$ form the partite sets of vertices in the product $C_4 \times K_n$. Also, for $1\leq i,j\leq n$ and $i\neq j$, $U^i \cup U^j$ induces $K_{4,4}-2I$, where $I$ is a $1$-factor of $K_{4,4}$. Now, each set $U^i \cup U^j$ is isomorphic to  $K_{4,4}-2I$. But $K_{4,4}-2I$ admits a $4$-cycle decomposition. Hence the proof.
\end{proof}
\begin{lem}\label{T:O7}
$C_6|C_6 \times K_2$
\end{lem}
\begin{proof}
Let the partite set of the bipartite graph $C_6 \times K_2$ be $\{u_1,u_2,...,u_6\}$,\\$\{v_1,v_2,...,$ $v_6\}$. We assume that the vertices having the same subscripts are the corresponding vertices of the partite sets. Now $C_6 \times K_2$ can be decomposed into $6$-cycles which are $\{u_1,v_2,u_3,v_4,u_5,v_6\}$ and $\{v_1,u_2,v_3,u_4,v_5,u_6\}$.
\end{proof}
\begin{thm}\label{T:O8}
For all $n$, $C_6|C_6 \times K_n$.
\end{thm} 
\begin{proof}
Let the partite set of the $6$-partite graph $C_6 \times K_n$ be $U=\{u_1,u_2,...,$\\$u_n\}$, $V=\{v_1,v_2,...,v_n\}$, $W=\{w_1,w_2,...,w_n\}$, $X=\{x_1,x_2,...,x_n\}$, $Y=\{y_1,y_2,...,y_n\}$ and $Z=\{z_1,z_2,...,z_n\}$: we assume that the vertices of $U,V,W,X,Y$ and $Z$ having the same subscripts are the corresponding vertices of the partite sets. Let $U^1=\{u_1,v_1,w_1,x_1,y_1,z_1\}$, $U^2=\{u_2,v_2,w_2,x_2,y_2,z_2\}$, ..., $U^n=\{u_n,v_n,w_n,x_n,y_n,z_n\}$ be the sets of these vertices having the same subscripts. By the definition of the tensor product, each $U^i$, $1 \leq i \leq n$ is an independent set and the subgraph induced by each $U^i \cup U^j$, $1\leq i,j\leq n$ and $i\neq j$ is isomorphic to $C_6 \times K_2$. Now by Lemma \ref{T:O7} the graph $C_6 \times K_2$ admits a $6$-cycle decomposition. This completes the proof.  
\end{proof}

\indent Furthermore, we quote the following result on decomposition of the tensor product of graphs into cycles of odd length. 
\begin{lem}\label{TT:O9}\cite{MP2}
For $k \geq 1$ and $m \geq 3$, $C_{2k+1}|C_{2k+1} \times K_m$ 
\end{lem}
\indent The next theorem is an extension of Lemma \ref{TT:O9}. 
\begin{thm}\label{TT:O10}
For $m \geq 3$ and $n \geq 2$, $C_m|C_m \times K_n$
\end{thm}
\begin{proof}
We shall split the proof of this lemma into two cases.\\
\\{\bf Case 1.}When $m=2k+1$, $k \geq 1$\\
The proof of this case is immediate from Lemma \ref{TT:O9}.\\
\\{\bf Case 2.}When $m=2k$, $k \geq 2$\\
Following from the definition of tensor product of graphs, let $U_1= \{u^1_1,u^2_1,u^3_1,...,u^m_1\}$, $U_2= \{u^1_2,u^2_2,u^3_2,...,u^m_2\} $,$...$, $U_n=\{u^1_n,u^2_n,u^3_n,...,u^m_n\}$ form the partite sets of vertices in the product $C_m \times K_n$. Now, for $1\leq i,j\leq n$ and $i\neq j$, the subgraph induced by $U_i \cup U_j$ is isomorphic to  $K_{m,m}-(m-2)I$, where $I$ is a $1$-factor of $K_{m,m}$. But $K_{m,m}-(m-2)I$ admits an $m$-cycle decomposition. Hence the proof.
\end{proof}

\section{ \bf Proof of the Main Theorem} \noindent \textbf{Proof of Theorem \ref{TT:O1}.} Assume that $C_4|K_m \times K_n$, for some $m$ and $n$ with $2 \leq m,n$. Then every vertex of $K_m \times K_n$ has even degree and $4$ divides the number of edges of $K_m \times K_n$. These two conditions translates to $(m-1)(n-1)$ being even and $8|mn(m-1)(n-1)$ respectively. Hence by the first fact, $m$ or $n$ has to be odd, i.e. has to be congruent to $1 \ or \ 3 \ or \ 5\ (\textrm{mod}\ 6)$. The second condition is satisfied precisely when one of the following holds. 
\ben  
\item  
 $n \equiv 0 \ (\textrm{mod}\ 4)$ and $m$ is odd,
\item
 $m \equiv 0 \ (\textrm{mod}\ 4)$ and $n$ is odd, or
 \item
 $m \ or \ n \equiv 1 \ (\textrm{mod}\ 4)$.
\een
 
Next we proceed to prove the sufficiency in two cases.\\
{\bf Case 1.} Since the tensor product is commutative, we may assume that $m$ is odd and so $m \equiv \ 1 \ or \ 3 \ or \ 5\ (\textrm{mod}\ 6)$. Suppose that $n \equiv 0 \ (\textrm{mod}\ 4)$.\\
{\bf Subcase 1.} Let $m \equiv 1 \ or \ 3\ (\textrm{mod}\ 6)$\\
Now since $m \equiv 1 \ or \ 3\ (\textrm{mod}\ 6)$ it implies that by Theorem \ref{TT:O5} $C_3|K_m$. Therefore, the graph $K_m \times K_n =((C_3 \times K_n)\oplus \cdots \oplus(C_3 \times K_n))$. But $n \equiv 0 \ (\textrm{mod}\ 4)$ therefore by Theorem \ref{TT:O7} we have that $C_4|C_3 \times K_n$. Hence $C_4|K_m \times K_n$.\\
{\bf Subcase 2.} Let $m \equiv 5(\textrm{mod}\ 6)$\\
By Lemma \ref{L:O1}, there exist positive integers $p$ and $q$ such that  $K_m$ is decomposable into $p$ $3$-cycles and $q$ $4$-cycles. Hence $K_m \times K_n$ has a decomposition into $p$ copies of $C_3 \times K_n$ and $q$ copies of $C_4 \times K_n$. By Theorem \ref{TT:O7} $C_4|C_3 \times K_n$ and also Lemma \ref{TT:O8} shows that $C_4|C_4 \times K_n$. Hence $C_4| K_m \times K_n$.  \\
{\bf Case 2.} By commutativity of the tensor product we assume that $m \equiv 1\ (\textrm{mod}\ 4)$. The graph $K_m \times K_n =((K_m \times K_2)\oplus \cdots \oplus(K_m \times K_2))$. Since $m \equiv 1\ (\textrm{mod}\ 4)$, by Corollary \ref{MAcor}, $C_4|K_{m,m}-I$, and $K_m \times K_2 \cong K_{m,m}-I$. Hence $C_4| K_m \times K_n$. This completes the proof. \ \ \ \ \ \ \ \ \ \ \ \ \ \ \ \ \ \ \ \ \ \ \ \ \ \ \  \  \ \ \ \ \ \ \ \ \ \ \ \ \ \ \ \ \ \ \  \ \ \ \ \ \ \ \ \ \ \ \ \ \ \ \ \ \ \ \ \ \ \  \ \ \ \ \ \ \ \ \ $\Box$ \\

\noindent
 \textbf{Proof of Theorem \ref{T:O1}.} 
Assume that $C_6|K_m \times K_n$ for some $m$ and $n$ with $2 \leq m,n$. Then every vertex of $K_m \times K_n$ has even degree and $6$ divides in the number of edges of $K_m \times K_n$. These two conditions translate to $(m-1)(n-1)$ being even and $6|m(m-1)n(n-1)$ respectively. Hence, by the first fact $m$ or $n$ has to be odd, i.e., has to be congruent to $1 \ or \ 3 \ or \ 5\ (\textrm{mod}\ 6)$. The second fact can now be used to show that they cannot both be congruent to $5\ (\textrm{mod}\ 6)$. It now follows that $m \equiv 1 \ or \ 3\ (\textrm{mod}\ 6)$ or $n \equiv 1 \ or \ 3\ (\textrm{mod}\ 6)$.\\
Conversely, let $m \equiv 1 \ or \ 3\ (\textrm{mod}\ 6)$.
By Theorem \ref{TT:O5}, $C_3|K_m$ and hence $K_m \times K_n =((C_3 \times K_n)\oplus \cdots \oplus(C_3 \times K_n))$. Since $C_6|C_3 \times K_n$ by Theorem \ref{T:O3}.\\
Finally, if $n \equiv 1 \ or \ 3\ (\textrm{mod}\ 6)$, the above argument can be repeated with the roles of $m$ and $n$ interchanged to show again that $C_6|K_m \times K_n$. This completes the proof.\\ 
\ \\
\noindent
\textbf{Proof of Theorem \ref{T:O2}.} 
Assume that $C_6|K_m-I \times K_n, m\geq 6$. Certainly, $6|mn(m-2)(n-1)$. But we know that if $6|m(m-2)$ then $6|mn(m-2)(n-1)$. But $m$ is even therefore $m \equiv 0 \ or \ 2\ (\textrm{mod}\ 6)$.\\ 
Conversely, let $m \equiv 0 \ or \ 2 \ (\textrm{mod}\ 6)$. Notice that for each $m$, $\frac{m(m-2)}{2}$ is a multiple of $3$. Thus by Theorem \ref{T:O5} $C_3|K_m -I$ and hence $K_m - I \times K_n =((C_3 \times K_n)\oplus \cdots \oplus(C_3 \times K_n))$. From Theorem \ref{T:O3} $C_6|C_3 \times K_n$. The proof is complete. 

\section{ \bf CONCLUSION}
\noindent
In view of the results obtained in this paper we draw our conclusion by the following remark and corollary. 
\begin{rem}
The product $K_m \times K_n$ can also be viewed as an even regular complete multipartite graph. So
by the conditions given in Theorem \ref{TT:O1} we have that every even regular complete multipartite graph $G$ is  $C_4$-decomposable if the number of edges of $G$ is divisible by $4$.
\end{rem}

\begin{cor}
For any simple graph $G$. If \\
\ben
\item
$C_3 |G$ then $C_6 |G \times K_n$, whenever $n \geq 2$.
\item
$C_6 |G$ then $C_6 |G \times K_n$, whenever $n \geq 2$.
\een
\end{cor}
\begin{proof}
We only need to show that $C_3|G$. Applying Theorem \ref{T:O3} gives the result.  
\end{proof}  

\begin{center}
\textbf{Comments}\\
\end{center}
\noindent
The results presented in this work are split in two different papers \cite{AO, OA}.  
%\begin{center} 
%\textbf{Acknowledgments}\\
%\end{center}
%\noindent
%The authors would like to appreciate the constructive criticisms which the main theorem of this work received from the anonymous referees, and for bringing to our notice the references \cite{KK, MA}, these efforts has tremendously improved the work.

\end{document}